\documentclass[11pt]{article}

\usepackage{stmaryrd}
\usepackage{textcomp}
\usepackage{enumerate}
\usepackage{setspace}
\usepackage{amssymb}
\usepackage{amsthm}
\usepackage{amsmath}
\usepackage{graphicx}
\usepackage{extarrows}
\usepackage{marvosym}
\usepackage{latexsym}
\usepackage{endnotes}
\usepackage{fontenc}
\usepackage[usenames,dvipsnames]{color}
\usepackage{authblk}

\usepackage[titletoc,toc,title]{appendix}
\usepackage{url}
\usepackage{tikz}
    
    \def\subdivisionscale{1.5}
    \usetikzlibrary{patterns,decorations.pathreplacing}

\newtheorem{theorem}{Theorem}
\newtheorem{lemma}[theorem]{Lemma}

\newtheorem{fact}[theorem]{Fact}
\newtheorem{prop}[theorem]{Proposition}
\newtheorem{observation}[theorem]{Observation}

\theoremstyle{definition}

\newtheorem{remark}[theorem]{Remark}
\newtheorem*{remark*}{Remark}

\newtheorem{example}[theorem]{Example}

\def\NN{{\mathbb N}}

\renewcommand{\geq}{\geqslant}

\renewcommand{\leq}{\leqslant}

\newcommand\cA{\mathcal{A}}

\newcommand\cH{\mathcal{H}}

\def\deg{\operatorname{deg}}
\def\gap{\operatorname{gap}}
\def\Ex{\operatorname{Ex}}

\newcommand{\limp}{{\rm limp}}

\begin{document}

\title{On the purity of minor-closed classes of graphs}

\author[1]{Colin McDiarmid}
\author[2]{Micha{\l} Przykucki}
\affil[1]{Department of Statistics, University of Oxford\thanks{Email: cmcd@stats.ox.ac.uk}}
\affil[2]{School of Mathematics, University of Birmingham\thanks{Email: m.j.przykucki@bham.ac.uk. During a large part of this project, the second author was affiliated with St~Anne's College and the Mathematical Institute of the University of Oxford.}}
\renewcommand\Authands{ and }
\date{\today}

\maketitle

\makeatletter{\renewcommand*{\@makefnmark}{}
\footnotetext{\textcopyright~2018. This manuscript version is made available under the CC-BY-NC-ND 4.0 License.}\makeatother}

\begin{abstract}
Given a graph $H$ with at least one edge, let $\gap_{H}(n)$ denote the maximum difference between the numbers of edges in two $n$-vertex edge-maximal graphs with no minor $H$.  We show that for exactly four connected graphs $H$ (with at least two vertices), the class of graphs with no minor $H$ is pure, that is, $\gap_{H}(n) = 0$ for all $n \geq 1$; and for each connected graph $H$ (with at least two vertices) we have the dichotomy that either $\gap_{H}(n) = O(1)$ or $\gap_{H}(n) = \Theta(n)$. Further, if $H$ is 2-connected and does not yield a pure class, then there is a constant $c>0$ such that  $\gap_{H}(n) \sim cn$.  We also give some partial results when $H$ is not connected or when there are two or more excluded minors. 
\end{abstract}


\section{Introduction}
\label{sec:intro}

We say that a graph $G$ contains a graph $H$ as a \emph{minor} if we can obtain a graph isomorphic to $H$ from a subgraph of $G$ by using edge contractions (discarding any loops and multiple edges, we are interested in simple graphs). A class of graphs $\cA$ is \emph{minor-closed} if for each $G \in \cA$, each minor $G'$ of $G$ is also in $\cA$. We say that $H$ is an \emph{excluded minor} for $\cA$ if $H$ is not in $\cA$ but each minor of $H$ (other than $H$ itself) is in $\cA$. Robertson and Seymour \cite{robertsonSeymour} showed that, for each minor-closed class $\cA$ of graphs, the set $\cH$ of excluded minors is finite. We say that $G$ is $H$-\emph{free} if it has no minor $H$; and given a set $\cH$ of graphs, $G$ is $\cH$-\emph{free} if it is $H$-free for all $H \in \cH$.  We denote the class of all $\cH$-free graphs by $\Ex(\cH)$, and write $\Ex(H)$ when $\cH$ consists of just the graph $H$.  

Observe that $\Ex(\cH)$ contains $n$-vertex graphs for each $n$ as long as each graph in $\cH$ has at least one edge.  Let us call such a set $\cH$ of graphs \emph{suitable} as long as it is non-empty.  We shall restrict our attention to suitable classes $\cH$, and be interested in the number of edges in edge-maximal $\cH$-free graphs. Given a graph $G$, let $v(G)$ denote the number of vertices and $e(G)$ the number of edges.  For all $n \geq 1$, let
\[
 E_\cH(n) = \{ e(G) : v(G) = n \text{ and } G \text{ is an edge-maximal $\cH$-free graph} \}.
\]
Also, let $M^+_\cH(n) = \max E_\cH(n)$, and $M^-_\cH(n) = \min E_\cH(n)$. Finally, let us define
\[
 \gap_{\cH}(n) = M^+_\cH(n) - M^-_\cH(n).
\]
As for $\Ex(H)$, we write $\gap_{H}(n)$ to denote $\gap_{\cH}(n)$ when $\cH$ consists of just the graph $H$. This is the case on which we focus.

The function $M^+_\cH(n)$ (sometimes in the form of $2M^+_\cH(n)/n$ to analyse the maximum average degree of graphs in $\cA=\Ex(\cH)$) has been studied extensively for various suitable sets $\cH$. Mader \cite{minorsLinear} showed that, given a graph $H$, there is a constant $c=c(H)$ such that $e(G) \leq c\, v(G)$ for each graph $G \in \Ex(H)$. Let us define $\beta_{\cH}$ by setting
\begin{equation} \label{eqn.betadef}
 \beta_{\cH} := \sup_{G \in \cA} \frac{e(G)}{v(G)} = \sup_{n \geq 1} \frac{M^+_{\cH}(n)}{n},
\end{equation}
noting that $\beta_{\cH}$ is finite.  Write $\beta_{H}$ when $\cH$ consists just of the graph $H$. Building on work of Mader \cite{maderComplete}, Kostochka \cite{kostochkaComplete} and Fernandez de la Vega \cite{fdlvComplete}, Thomason \cite{completeMinors} showed that, for each positive integer $r$, we have $M^+_{K_r}(n) \sim \beta_{K_r} n$ as $n \to \infty$; and $\beta_{K_r} \sim \alpha \,  r \sqrt{\log r}$ as $r \to \infty$, where $\alpha \approx 0.319$. The value of $\beta_H$ for dense graphs $H$ was studied by Myers and Thomason \cite{nonCompleteMinors}. Reed and Wood \cite{sparseMinors} analysed this parameter for sparse forbidden minors $H$. Cs\'oka, Lo, Norin, Wu and Yepremyan \cite{disconnectedMinors} focused on $H$ being a union of disjoint cycles and, more generally, of disjoint $2$-connected graphs.

Much less is known about the function $M^-_\cH(n)$ and, consequently, about $\gap_{\cH}(n)$, for a suitable set $\cH$.  From Mader's result it follows that we always have $\gap_{\cH}(n) = O(n)$. We say that the class $\cA = \Ex(\cH)$ is \emph{pure} if we have $\gap_{\cH}(n) = 0$ for each positive integer $n$.  For example, the class $\Ex(K_3)$ of forests is pure, since the $n$-vertex edge-maximal forests are the trees, each with $n-1$ edges. Our first main theorem is:
\begin{theorem}
\label{thm:connpure}
The connected graphs $H$ on at least two vertices such that the class $\Ex(H)$ is pure are precisely the complete graphs $ K_2, K_3, K_4$ and the 3-vertex path $P_3$. 
\end{theorem}

We say that $\Ex(\cH)$ is \emph{near-pure} if it is not pure, but we still have $\gap_{\cH}(n) = O(1)$. Also, we define the `linear impurity parameter'
\[
 \limp(\cH) = \liminf_{n \to \infty} \frac{\gap_{\cH}(n)}{n};
\]
and we say that $\Ex(\cH)$ is \emph{linearly impure} if $\limp(\cH)>0$. Our second main result shows that all connected graphs $H$ fall into one of only three categories according to the purity of the class $\Ex(H)$.

\begin{theorem}
\label{thm:threeClasses}
For each connected graph $H$ on at least two vertices, the class of $H$-free graphs is either pure, near-pure or linearly impure.
\end{theorem}
In other words, Theorem \ref{thm:threeClasses} says that it is not possible for the impurity of a class of $H$-free graphs to be unbounded but not grow linearly fast in $n$. We have seen in Theorem~\ref{thm:connpure} that if $H$ is $K_3$ or $K_4$ then $\gap_H(n)=0$ for each $n$, and $\limp(H)=0$. More generally, whenever $H$ is $2$-connected, $\gap_{H}(n)/n$ tends to limit, so the `liminf' in the definition of limp could be replaced by the more satisfactory `lim' (see also Theorem~\ref{thm:AddableLimit}).
\begin{theorem}
\label{thm:PositiveLimit}
Let $H$ be a 2-connected graph other than $K_3$ or $K_4$. Then, as $n \to \infty$,
\begin{equation}
  \label{eq:positiveLimit}
  \frac{\gap_{H}(n)}{n} \; \to \; \limp(H) \: >0.
\end{equation}
\end{theorem}

An important example of a pure minor-closed class is the class of planar graphs. Indeed, for each $n \geq 3$, all $n$-vertex edge-maximal graphs $G$ embeddable in the plane are triangulations, satisfying $e(G) = 3n-6$. However, somewhat surprisingly, it is not the case that a similar statement holds for graphs embeddable in the torus: it was shown in~\cite{torusNonTriangulation} that a complete graph on $8$ vertices with the edges of a $5$-cycle $C_5$ removed (thus containing $23$ edges) is an edge-maximal graph embeddable in the torus, while each $8$-vertex triangulation of the torus, by Euler's formula, contains $24$ edges. However, for every surface $S$, the (minor-closed) class of graphs embeddable in $S$ is pure or near-pure, as shown by McDiarmid and Wood~\cite{embeddableImpurity}.

At this point, let us check that the four connected graphs listed in Theorem~\ref{thm:connpure}, namely $K_2$, $K_3$, $K_4$ and $P_3$, give rise to pure $H$-free classes of graphs. The case of $K_2$ is trivial, as $\Ex(K_2)$ consists of the graphs without edges. We already noted that the class $\Ex(K_3)$ of forests is pure.  If $H=P_3$, the path on $3$ vertices, then the $n$-vertex edge-maximal $H$-free graphs are the maximal matchings, each with $\lfloor n/2 \rfloor$ edges. Finally, the class $\Ex(K_4)$ is the class of series-parallel graphs, which is also the class of graphs of treewidth at most 2.  For each $n \geq 2$ each $n$-vertex edge-maximal such graph has exactly $2n-3$ edges. In fact, for each fixed $k \geq 1$ the edge-maximal graphs of treewidth at most $k$ are the $k$-trees, and each $n$-vertex $k$-tree has $kn-\binom{k+1}{2}$ edges for $n \geq k$ (and $\binom {n}{2}$ for $n<k$).  Thus for each $k \geq 1$ the class of graphs of treewidth at most $k$ is pure.  We will have to work much harder to prove that the four graphs listed are the only connected graphs $H$ for which $\Ex(H)$ is pure!

The rest of the paper is organised as follows. In the next section we introduce addable graph classes, and prove a general limiting result, Theorem~\ref{thm:AddableLimit}, which yields the `limit part' of Theorem~\ref{thm:PositiveLimit}.  We also sketch a useful consequence of purity or near-purity for such a class of graphs. In Section~\ref{sec.noleaf} we show that for each connected graph $H$ with no leaf (that is, with minimum degree $\delta(H) \geq 2$), if $H$ is not $K_3$ or $K_4$ then $\Ex(H)$ is linearly impure.  This is a step towards proving both Theorems~\ref{thm:connpure} and~\ref{thm:threeClasses}, and together with Theorem~\ref{thm:AddableLimit} proves Theorem~\ref{thm:PositiveLimit}, concerning a 2-connected graph $H$. In Section~\ref{sec.leaf} we complete the proof of Theorem~\ref{thm:threeClasses}, showing that for a connected excluded minor there are only the three possibilities of purity, near-purity or linear impurity. In Section~\ref{sec:allNonPure} we complete the proof of Theorem~\ref{thm:connpure}, showing that only four connected graphs $H$ give rise to pure $H$-free classes. In Section \ref{sec:generalizations} we give some extensions of our results to suitable sets $\cH$ of two or more excluded graphs, and to forbidding disconnected graphs; and finally we propose some natural open problems.


\section{Addable graph classes}
\label{sec:addable}

In this section we introduce addable graph classes.  We show that, for an addable minor-closed class $\cA$ of graphs with suitable set $\cH$ of excluded minors, $ \gap_{\cH}(n)/n $ tends to a limit, and we identify that limit as a difference of two terms (see~(\ref{eq:addableLimit})).  Finally we describe a consequence of purity or near-purity for growth constants when we have a given average degree. \smallskip

We say that a graph class $\cA$ is \emph{addable} when
\begin{enumerate}
 \item $G \in \cA$ if and only if every component of $G$ is in $\cA$ (following Kolchin \cite{randomMappings}, if $\cA$ satisfies this property we call it \emph{decomposable}), and
 \item whenever $G \in \cA$ and $u,v$ belong to different components of $G$ then the graph obtained from $G$ by adding the edge $\{u,v\}$ is also in $\cA$ (following \cite{randomMinorClosed}, such a class $\cA$ is called \emph{bridge-addable}).
\end{enumerate}
A minor-closed class is decomposable if and only if each excluded minor is connected, and it is addable if and only if each excluded minor is $2$-connected. For example, the classes of forests ($\Ex(K_3)$), series-parallel graphs ($\Ex(K_4)$), and planar graphs $(\Ex(\{ K_5, K_{3,3} \})$) are each addable.

The following general limiting result shows that in the addable case, the `liminf' in the definition of limp can be replaced by `lim'.
\begin{theorem}
\label{thm:AddableLimit}
Let $\cA$ be an addable minor-closed class of graphs, with suitable set $\cH$ of excluded minors. Then, as $n \to \infty$,
\begin{equation}
 \label{eq:AddableLimit}
 \frac{\gap_{\cH}(n)}{n} \to \limp(\cH).
\end{equation}
\end{theorem}

To prove this result, we use two lemmas, treating $M^+_{\cH}(n)$ and  $M^-_{\cH}(n)$ separately.  Recall that $\beta_{\cH}$ was defined in~(\ref{eqn.betadef}).  In the following lemma, it is easy to see that $\beta_{\cH} \geq 1$, since $\cA$ contains all the forests.

\begin{lemma}
\label{lem.decomp}
Let $\cA$ be a decomposable minor-closed class of graphs, with suitable set $\cH$ of excluded minors.  Then
\[
 \frac1{n} M^+_\cH(n) \to \beta_{\cH} \;\; \mbox{ as } n \to \infty.
\]
\end{lemma}
\noindent
\begin{proof}
Denote $M^+_\cH(n)$ by $f(n)$.  For $i=1,2$ let $n_i$ be a positive integer and let $G_i \in \cA_{n_i}$ satisfy $e(G_i)=f(n_i)$. Since the disjoint union $G_1 \cup G_2$ is in $\cA_{n_1+n_2}$ we have
\[
 f(n_1+n_2) \geq f(n_1)+f(n_2);
\]
that is, $f$ is superadditive.  Hence by Fekete's Lemma (see for example van Lint and Wilson \cite{combinatoricsCourse})
\[
 \frac{f(n)}{n} \, \to \; \sup_k \frac{f(k)}{k} = \beta_\cH \;\; \mbox{ as } n \to \infty.
\]
\end{proof}

\begin{lemma}
  Let $\cA$ be an addable minor-closed class of graphs, with suitable set $\cH$ of excluded minors.  Then there is a constant $\beta^-_\cH \geq 1$ such that
\[
 \frac1{n} M^-_\cH(n) \to \beta^-_{\cH} \;\; \mbox{ as } n \to \infty.
\]
\end{lemma}
\begin{proof}
Let $h = \min \{ v(H) : H \in \cH \},$ and note that $h \geq 3$.  Consider the function $f(n) = M^-_{\cH}(n) + (h-2)^2$.  Note that each edge-maximal graph in $\cA$ is connected, so $f(n) \geq n$ for each $n$. 
Let $\beta^-_\cH = \inf_{k} f(k)/k \geq 1$. We claim that $f(n)$ is subadditive (that is $f(a+b) \leq f(a) + f(b)$), so by Fekete's Lemma, as $n \to \infty$ we have $f(n)/n \to \beta^-_\cH$ and thus also $M^-_{\cH}(n)/n \to \beta^-_\cH$.

 It remains to establish the claim that $f$ is subadditive. Let $n_1, n_2 \geq 1$ and let $G_1, G_2$ be edge-maximal $\cH$-free graphs with $v(G_1) = n_1$, $v(G_2) = n_2$, and such that $e(G_1)  = M^-_{\cH}(n_1), e(G_2) = M^-_{\cH}(n_2)$. Note that $G_1$  and $G_2$ are connected.
 
As in the proof of the last lemma, the disjoint union $G = G_1 \cup G_2$ is $\cH$-free. It will be enough to show that we cannot add more than $(h-2)^2$ edges to $G$ without creating an $H$-minor for some $H \in \cH$. Indeed, let $u_1 \neq v_1$ be in $V(G_1)$ and let $u_2,v_2$ be in $V(G_2)$, and assume that we can (simultaneously) add the edges $\{u_1,u_2 \}$ and $\{v_1,v_2 \}$ to $G$ without creating any $H$-minor. Then the edge $\{u_1, v_1\}$ must be present in $G_1$ since otherwise, after adding $\{u_1,u_2 \}$ and $\{v_1,v_2 \}$ to $G$, by the connectedness of $G_2$ there is a path between $u_1$ and $v_1$ that uses only vertices in $G_2$, and we may contract this path to an edge between $u_1$ and $v_1$: this would necessarily create an $H$-minor for some $H \in \cH$ by the edge-maximality of $G_1$.
 
 Hence if we add edges to $G$ without creating any $H$-minor then the vertices in $G_1$ incident to the edges that we add must induce a clique in $G_1$, with an analogous statement holding for $G_2$. By the definition of $h$, these cliques can have size at most $h-2$ (if there were an $(h-1)$-clique in $G_1$ say, and we contracted $G_2$ to a single vertex, we would obtain an $h$-clique), hence we can add at most $(h-2)^2$ edges. Consequently,
 \begin{align*}
  f(n_1+n_2) & = M^-_{\cH}(n_1+n_2) + (h-2)^2 \\
         & \leq \left ( M^-_{\cH}(n_1) + M^-_{\cH}(n_2) + (h-2)^2 \right ) + (h-2)^2 \\
         & = f(n_1)+f(n_2	).       
 \end{align*}
 Thus $f(n)$ is subadditive, and the proof is complete.
\end{proof}
 The last two lemmas show that, if $\cA$ is an addable minor-closed class of graphs with suitable set $\cH$ of excluded minors, then
\begin{equation}
  \label{eq:addableLimit}
  \frac{\gap_{\cH}(n)}{n} \to \beta_\cH - \beta^-_\cH \;\; \mbox{ as } n \to \infty.
\end{equation}
Thus
$ \limp(\cH) = \beta_\cH - \beta^-_\cH$,
and
$\frac{\gap_{\cH}(n)}{n} \to  \limp(\cH)$ as $n \to \infty$,
which completes the proof of Theorem~\ref{thm:AddableLimit}.

We close this section by sketching a useful consequence of purity or near-purity.  Let $\cA$ be a minor-closed class of graphs, with non-empty set $\cH$ of excluded minors. Let $\cA_n$ denote the set of graphs in $\cA$ on vertex set $[n]=\{1,2,\ldots,n\}$, let $a_n = |\cA_n|$, and let
\[
 \gamma(\cA) = \limsup_{n \to \infty} \left ( \frac{a_n}{n!} \right )^{1/n}.
\]
Norine, Seymour, Thomas and Wollan \cite{properSmall} (see also Dvo\v{r}\'ak and Norine \cite{smallClasses}) showed that 
$\gamma(\cA) < \infty$. Now suppose that $\cA$ is addable, that is, the excluded minors are $2$-connected. Then (see, for example \cite{randomMinorClosed}), $( a_n / n! )^{1/n}$ converges to $\gamma(\cA)$ and we say that $\cA$ has \emph{growth constant} $\gamma(\cA)$. Defining $a_{n,q} = |\cA_{n,q}|$ to be the number of graphs in $\cA_n$ with $\lfloor qn \rfloor$ edges, following the methods in Gerke, McDiarmid, Steger and Wei{\ss}l \cite{planarSoda} it can be shown that $(a_{n,q} / n!)^{1/n}$ tends to a limit $\gamma(\cA,q)$. If $\cA$ is pure or near-pure then, again following the analysis in \cite{planarSoda}, we may see that $\gamma(\cA,q)$ as a function of $q$ is log-concave, and hence continuous, for $q \in (1, \beta_\cH)$.


\section{Purity and linear impurity: excluding a leafless graph}
\label{sec.noleaf}

In this section we prove the following lemma, which shows linear impurity for some excluded minors~$H$.  It is a step towards proving both Theorems~\ref{thm:connpure} and~\ref{thm:threeClasses}, and together with Theorem~\ref{thm:AddableLimit} immediately yields Theorem~\ref{thm:PositiveLimit}.
\begin{lemma} \label{lem.noleaf}
  Let $H$ be a connected graph with $\delta(H) \geq 2$, other than $K_3$ and $K_4$.  Then $H$ is linearly impure.
\end{lemma}

We shall often use the following fact proved by Sylvester in $1884$.
\begin{fact}
 \label{fact:frobeniusNumber}
 Let $a_1, a_2$ be a pair of positive coprime integers. Then for every integer $N > a_1 a_2 - a_1 - a_2$ there are some non-negative integers $b_1, b_2$ such that
 \[
  N = a_1 b_1 + a_2 b_2.
 \]
\end{fact}

Let us call a vertex $v$ in a connected $h$-vertex graph $H$ a \emph{strong separating vertex} if each component of $H-v$ has at most $h-3$ vertices (so $v$ is a separating vertex which does not just cut off a single leaf). In order to prove Lemma~\ref{lem.noleaf} we first consider complete graphs, and then non-complete graphs with no leaves. In the next lemma we deal with complete graphs.
\begin{lemma}
 \label{lem:completeImpure}
 For each $r \geq 5$ the class of $K_r$-free graphs satisfies $\limp(K_r) \geq \frac76$.
\end{lemma}
\begin{proof}
 We prove the lemma by induction on $r$. First, let $r = 5$. Wagner \cite{wagnerK5} showed that any edge-maximal $K_5$-free graph on at least $4$ vertices can be constructed recursively, by identifying edges or triangles, from edge-maximal planar graphs (i.e., triangulations) and copies of the Wagner graph (recall that the Wagner graph is formed from the cycle $C_8$ by joining the four opposite pairs of vertices, hence it has $8$ vertices and $12$ edges). If $n = 6k+2$, we can take $G_1$ to be an arbitrary plane triangulation on $n$ vertices with $e(G_1) = 3n-6 = 18k$. We then take $G_2$ to be a clique-sum of $k$ copies of the Wagner graph $W_8$ that all overlap in one common edge. Then $e(G_2) = 11k+1$ and
 \[
  \frac{e(G_1) - e(G_2)}{n} = \frac{7k-1}{6k+2} \to 7/6
 \]
 as $k \to \infty$. For general $n$ we can modify the construction of $G_2$ by taking a clique-sum of $k$ copies of $W_8$ and a triangulation on $3 \leq m \leq 7$ vertices (in fact, by the above characterisation of the edge-maximal $K_5$-free graphs, it is easy to check that $\limp(K_5) = \frac76$). Therefore the lemma holds for $r=5$.
 
 The statement for $r+1$ follows from the statement for $r$ by observing that if we take any edge-maximal $K_r$-free graph $G$, add to it one vertex and connect it to all vertices of $G$, then the resulting graph is edge-maximal $K_{r+1}$-free.
\end{proof}

\begin{remark}
Recall from \cite{completeMinors} that $M^+_{K_r}(n) \sim \alpha \, r \sqrt{\log r} \, n$ for $\alpha \approx 0.319$, while the constructions in Lemma \ref{lem:completeImpure} have both $e(G_1)$ and $e(G_2)$ that grow linearly with $r$. Thus we see that $\limp(K_r) \sim \alpha \, r \sqrt{\log r}$.
\end{remark}
\smallskip

We next consider connected graphs that are not complete but do not have any leaves. We say that $G$ has \emph{connectivity $k$} if $k$ is the minimum size of a vertex cut of $G$ (except that, for $n \geq 2$, $K_n$ has connectivity $n-1$).  Also, we say that $G$ is $j$-connected if $G$ has connectivity at least $j$.  Recall that $\delta(G)$ denotes the minimum degree and, for $u \in V(G)$, let
\[
 N(u) = \{ v \in V(G) : \{u,v\} \in E(G)\}
\]
denote the neighbourhood of $u$ in $G$. The following simple fact will be very useful to us.  

\begin{fact}
\label{fact:highlyConnected}
 Let $G$ be a non-complete graph on $n$ vertices with $\delta(G) = \delta$. Then $G$ has connectivity at least $2\delta-n+2$.
\end{fact}
\begin{proof}
 Let $u$ and $v$ be a non-adjacent pair of vertices.  Then
 \[ 
  2 \delta \leq \deg(u) + \deg(v) = |N(u) \cup N(v)| + |N(u) \cap N(v)| \leq n -2+ |N(u) \cap N(v)|,
 \]
so $u$ and $v$ have at least $2\delta-n+2$ common neighbours, and any vertex cut separating $u$ and $v$ must contain all of these vertices.
\end{proof}

\begin{lemma}
\label{lem:noLeaves}
 Let $H$ be a connected non-complete graph on $h \geq 4$ vertices with $\delta: = \delta(H) \geq 2$. Then the class of $H$-free graphs satisfies $\limp(H) \geq \frac1{2h}$.
\end{lemma}
\begin{proof}
Since $H$ is connected, $H$ has connectivity $k$ for some $k \geq \max \{2\delta-h+2, 1\}$. We first show that for all $m \geq 1$ there exist two graphs $G_1, G_2$, both on
\[
 n = (h-k)(h-k+1) m+k-1
\]
vertices, that are edge-maximal $H$-free and such that
\[
 e(G_1) - e(G_2) \geq \frac{(h-k)m}{2} = (1+o(1))\frac{n}{2(h-k+1)}.
\]

We construct the ``dense'' graph $G_1$ as follows. We take $(h-k+1)m$ copies of $K_{h-1}$ that all overlap in a fixed set of $k-1$ vertices. Clearly $G_1$ is $H$-free since $H$ has connectivity $k$ and trying to fit an $H$-minor in $G_1$ we would need to find it across more than one of the copies of $K_{h-1}$. Also, $G_1$ has $(h-k)(h-k+1) m+k-1$ vertices and
\[
\begin{split}
e(G_1) & = (h-k+1) m \left ( \binom{h-k}{2} + (h-k)(k-1) \right ) + \binom{k-1}{2} \\
       & = (h-k) m (h-k+1) \frac{h+k-3}{2} + \binom{k-1}{2}.
\end{split}
\]

We construct the ``sparse'' graph $G_2$ similarly. We start by taking $(h-k)m$ copies of $K_{h-1}$ that all overlap in a fixed set $I$ of $k-1$ vertices. The resulting graph $G'_2$ has $(h-k)^2 m + k-1$ vertices, i.e., $(h-k)m$ fewer that $G_1$. We complete the construction of $G_2$ by adding these $(h-k)m$ missing vertices and joining each of them to $\delta-1$ vertices in a distinct copy of $K_{h-1}$ in such a way that the neighbourhood of each new vertex does not contain the whole of $I$ (see Figure \ref{figure:noLeaves}).  Note that $G_2$ is $H$-free: for if $G_2$ had a minor $H$ then so would $G'_2$ (since vertices $v$ of degree $< \delta(H)$ with $N(v)$ complete are redundant), and we may see as for $G_1$ that $G_2'$ has no minor $H$.  We have
\[
\begin{split}
e(G_2) & = (h-k) m \left ( \binom{h-k}{2} + (h-k)(k-1) \right ) + \binom{k-1}{2} + (h-k)m(\delta-1) \\
       & = (h-k) m \left ( (h-k) \frac{h+k-3}{2} + \delta-1 \right ) + \binom{k-1}{2}.
\end{split}
\]

\begin{figure}[htb] \centering
  \begin{tikzpicture}
    \tikzstyle{vertex}=[draw,shape=circle,minimum size=5pt,inner sep=0pt]

    \draw (2,0) ellipse (2cm and 1cm);
    \draw (0,0) ellipse (2cm and 1cm);
    \draw (1,1) ellipse (1cm and 2cm);
    \draw (1,-1) ellipse (1cm and 2cm);
    \draw[black] (-1,0) node {$h-k$};
    \draw[black] (1,2) node {$h-k$};
    \draw[black] (1,-2) node {$h-k$};
    \draw[black] (3,0) node {$h-k$};
    \draw[black] (1,0) node {$k-1$};
    
    \foreach \name/\x/\y in {1/-2/-2, 2/3/-3, 3/4/2, 4/-1/3}
      \node[vertex] (P-\name) at (\x,\y) {~};
    \foreach \x/\y in {-1/-0.45, -0.75/-0.5, -0.5/-0.55}
      \draw (P-1) -- (\x,\y);
    \foreach \x/\y in {1.65/-2, 1.7/-1.75, 1.75/-1.5}
      \draw (P-2) -- (\x,\y);
    \foreach \x/\y in {3/0.45, 2.75/0.5, 2.5/0.55}
      \draw (P-3) -- (\x,\y);
    \foreach \x/\y in {0.35/2, 0.3/1.75, 0.25/1.5}
      \draw (P-4) -- (\x,\y);
    
  \end{tikzpicture}
  \caption{Graph $G_2$ as defined in Lemma \ref{lem:noLeaves}.}
  \label{figure:noLeaves}
\end{figure}
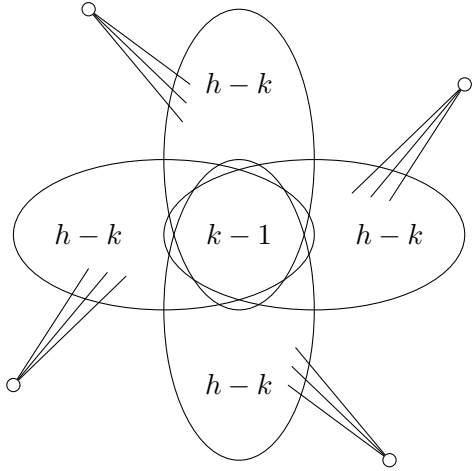

Consequently,
 \[
 e(G_1) - e(G_2) = (h-k) m \left ( \frac{h+k-3}{2} - \delta + 1 \right ).
\]
By Fact \ref{fact:highlyConnected} we have $h+k-3 \geq 2 \delta - 1$ hence
 \[
 e(G_1) - e(G_2) \geq \frac {(h-k) m}{2} = \frac{n-k+1}{2(h-k+1)} \sim \frac{n}{2(h-k+1)}.
\]

To show that $G_2$ is edge-maximal $H$-free, assume that we add an edge $e$ to $G_2$. If $e$ connects vertices not in $I$ in two distinct copies of $K_{h-1}$, then by contracting it we obtain two copies of $K_{h-1}$ that overlap in $k$ vertices and the resulting graph contains $H$ as a subgraph because $H$ has connectivity $k$. If $e$ connects a vertex $v$ of degree $\delta-1$ to a vertex in the copy of $K_{h-1}$ that contains the whole of $N(v)$ then this graph contains $H$ as a subgraph because now $\deg(v) = \delta = \delta(H)$. If finally $e$ connects a vertex $v$ of degree $\delta-1$ to another vertex $u$ that either has degree $\delta-1$ or is located in some other copy of $K_{h-1}$ then we can contract the path between $u$ and a vertex in $I \setminus N(v)$. The resulting graph again contains $H$ as a subgraph, because now $\deg(v) = \delta = \delta(H)$. 

To complete the proof of the lemma we observe that $h-k$ and $h-k+1$ are coprime. Thus by Fact \ref{fact:frobeniusNumber} for all $n$ large enough we can build approximations $G'_1, G'_2$ of the above graphs $G_1, G_2$ using the building blocks described above ($K_{h-1}$, and $K_{h-1}$ plus a vertex of degree $\delta-1$), with $\frac{e(G'_1) - e(G'_2)}{n} \to \frac{1}{2(h-k+1)} \geq \frac{1}{2h}$.
\end{proof}

At this stage, we have seen by Lemmas~\ref{lem:completeImpure} and~\ref{lem:noLeaves} that, if the connected graph $H$ has $\delta(H) \geq 2$ and $H$ is not $K_3$ or $K_4$, then $\limp(H)>0$;  that is, we have proved Lemma~\ref{lem.noleaf}.  Now Theorem~\ref{thm:PositiveLimit} follows from Theorem~\ref{thm:AddableLimit}.


\section{Purity, near-purity and linear impurity: excluding a graph with a leaf}
\label{sec.leaf}

In this section we complete the proof of Theorem~\ref{thm:threeClasses}, which says that for a connected excluded minor $H$ there are only the three possibilities of purity, near-purity or linear impurity for $\Ex(H)$.  We first deal quickly with graphs $H$ which have a strong separating vertex, treating the claw graph $K_{1,3}$ separately in Observation \ref{obs:claw}; and then we consider graphs $H$ with at least one leaf and no strong separating vertex.

\begin{lemma}
 \label{lem:separatingVertex}
 Let $H$ be a connected graph on $h \geq 5$ vertices which contains a strong separating vertex~$v$. Then the class of $H$-free graphs satisfies $\limp(H) \geq \frac12$.
\end{lemma}

\begin{proof}
 The construction here is very simple. For $m \geq 1$, let $G_1$ consist of $(h-2)m$ disjoint copies of $K_{h-1}$ and let $G_2$ consist of $(h-1)m$ disjoint copies of $K_{h-2}$. Both graphs contain $n = (h-1)(h-2)m$ vertices and are trivially $H$-free. They are edge-maximal $H$-free because whenever we add an edge $e$ to either $G_1$ or $G_2$, we can then contract it and identify the resulting common vertex of two cliques of size either $h-1$ or $h-2$ with $v$. The resulting graph contains $H$ as a subgraph because $h \geq 5$ and consequently $h-2 +h- 3 \geq h$.
 
 We clearly have $e(G_1) = (h-1)(h-2)^2 m/2$ and $e(G_2) = (h-1)(h-2)(h-3) m/2$. Hence
 \[
  e(G_1) - e(G_2) = \frac{(h-1)(h-2) m}{2} = \frac{n}{2}.
 \]
 The construction for general $n$ follows easily from Fact \ref{fact:frobeniusNumber} since $h-1$ and $h-2$ are coprime.
\end{proof}

\begin{observation}
\label{obs:claw}
The only connected graph on $h=4$ vertices with a strong separating vertex is the claw $K_{1,3}$. The class of $K_{1,3}$-free graphs is not pure, since for all $n \geq 4$ the cycle $C_n$ and the union of a cycle $C_{n-1}$ and an isolated vertex are edge-maximal $K_{1,3}$-free with $n$ and $n-1$ edges respectively.

However, this class is near-pure with $\gap_{K_{1,3}}(n) = 1$ for all $n \geq 4$. Indeed, note that any connected component of an edge-maximal $K_{1,3}$-free graph $G$ on $n$ vertices is either a cycle, an edge or an isolated vertex. Moreover, $G$ can have at most one component of size less than $3$ to preserve edge-maximality. Hence $G$ must have either $n$ or $n-1$ edges.
\end{observation}

For the rest of this section we consider the case when the connected graph $H$ on $h$ vertices has at least one leaf and has no strong separating vertex. We say that a connected graph $G$ is \emph{leaf-and-edge-maximal $H$-free} if $G$ is edge-maximal $H$-free and attaching a new leaf to an arbitrary vertex of $G$ creates an $H$-minor.

\begin{lemma}
\label{lem:IOmaximal}
Suppose that the connected graph $H$ has a leaf, and the class of $H$-free graphs is not linearly impure.  Then each leaf-and-edge-maximal $H$-free graph $G$ satisfies $e(G)/v(G) = (h-2)/2$; and each $H$-free graph $G$ satisfies $e(G)/v(G) \leq (h-2)/2$.
\end{lemma}
\begin{proof}
Indeed, if there existed two leaf-and-edge-maximal $H$-free graphs $G_1, G_2$ with $\frac{e(G_1)}{v(G_1)} >  \frac{e(G_2)}{v(G_2)}$ then we could trivially construct two arbitrarily large edge-maximal $H$-free graphs with the same number of vertices: $G'$ consisting of disjoint copies of $G_1$, and $G''$ consisting of disjoint copies of $G_2$, such that
\[
e(G')-e(G'') = \left ( \frac{e(G_1)}{v(G_1)} - \frac{e(G_2)}{v(G_2)} \right ) v(G').
\]
Further, to handle general $n$, to both $G'$ and $G''$ we could add a union of at most $\frac{\max\{v(G'),v(G'')\}}{h-1}$ disjoint copies of $K_{h-1}$ and a $K_{i}$ for some $1 \leq i \leq h-2$, keeping the graph edge-maximal $H$-free.

The claim now follows from the observation that, since $H$ has a leaf,  $K_{h-1}$ is always a leaf-and-edge-maximal $H$-free graph.  The second statement in the lemma follows similarly, by taking $G_1$ as $G$ and $G_2$ as $K_{h-1}$.
\end{proof}

\begin{observation}
\label{obs:atMostOneNotIO}
Suppose that $H$ has no strong separating vertex. Then any edge-maximal $H$-free graph contains at most one component that is not leaf-and-edge-maximal $H$-free. Otherwise we could connect two such components by a suitably attached edge, and the resulting graph would still be $H$-free because $H$ has no strong separating vertex and the components we started with were not leaf-and-edge-maximal $H$-free.
\end{observation}

The next lemma is the final step towards proving Theorem \ref{thm:threeClasses}.

\begin{lemma}
\label{lem:noGreyTerritory}
Let $H$ be a graph on $h$ vertices that is connected, has at least one leaf and has no strong separating vertex. If there exists $n > 0$ and two edge-maximal $H$-free graphs on $n$ vertices $G_1, G_2$ such that
\begin{equation}
\label{eq:boundOnM}
e(G_1) - e(G_2) \geq M = \frac{h-2}{2} + 2\beta_H^2+1
\end{equation}
then the class of $H$-free graphs is linearly impure.
\end{lemma}
\begin{proof}
Assume for a contradiction that the class of $H$-free graphs is not linearly impure. Let $G_1$ and $G_2$ be edge-maximal $H$-free graphs on the same vertex set such that $e(G_1)-e(G_2) \geq M$, for $M$ as in \eqref{eq:boundOnM} (we observe that $\beta_H \leq \beta_{K_h})$. If all components of $G_2$ were leaf-and-edge-maximal $H$-free then by Lemma \ref{lem:IOmaximal} we would have some component $C$ of $G_1$ that was $H$-free and satisfied $e(C)/v(C) > (h-2)/2$, which contradicts Lemma~\ref{lem:IOmaximal}.

Hence, by Observation \ref{obs:atMostOneNotIO}, $G_2$ has exactly one component $C$, with $|C|=c$, that is edge-maximal $H$-free, but not leaf-and-edge-maximal $H$-free. By Lemma~\ref{lem:IOmaximal} we have
\[
 e(C) \leq \frac{h-2}{2}c - M.
\]
Let $A$ be the set of all vertices $v$ in $C$ such that attaching a leaf to $v$ does not create an $H$ minor, and let $a = |A| \geq 1$. Clearly the graph induced by $A$ must be $H$-free so the set $A$ induces at most $\beta_H a$ edges. Let $v$ be a vertex in $A$ with the minimum number of neighbours in $A$: clearly, $\deg_A(v) \leq 2 \beta_H$.

Let $n = (c-1)m+1+t(h-1)+s$, where $t \geq 0$, $0 \leq s \leq h-2$, and $t(h-1)+s < c-1$. We take $m$ isomorphic copies of $C$ and turn them into one connected graph on $n' = (c-1)m+1$ vertices and $me(C)$ edges by identifying the vertices $v$ in all these copies into one vertex (still called $v$). Next, we add a copy of $K_{s}$ and join all (if $s \leq h-3$) or one (if $s = h-2$) of its vertices to $v$ by an edge. Finally, to this graph we add $t$ disjoint copies of $K_{h-1}$. The resulting graph $G$ on $n$ vertices is $H$-free by the definition of $A$ and by the fact that $H$ has no strong separating vertex (this latter property is the reason why we can join all the vertices of $K_s$ with $v$ if $s \leq h-3$).

We do not know if this graph is edge-maximal $H$-free. However, observe that we can only add edges to $G$ between distinct copies of the set $A$, or between one of the copies of $A$ and the clique $K_{s}$, or between $v$ and the clique $K_s$ (if $s=h-2$). Moreover, we are not allowed to add edges incident to vertices in $A$ that are not adjacent to $v$. Indeed, assume that we add an edge $\{u,w\}$ such that $u \notin N(v)$. Then by contracting a path from $w$ to $v$ (recall that $C$ is connected) we ``add'' the edge $\{u,v\}$ to a copy of $C$ which creates an $H$ minor by the edge-maximality of $C$ (see Figure \ref{figure:noGreyTerritory}). Hence there are at most $2 \beta_H m + s$ vertices other than $v$ between which we can add edges and keep the graph $H$-free.

\begin{figure}[htb] \centering
  \begin{tikzpicture}
    \tikzstyle{vertex}=[draw,shape=circle,minimum size=3pt,fill, inner sep=0pt]

    \draw[thick] (2.5,0) ellipse (2cm and 0.5cm);
    \draw[thick] (-0.5,0) ellipse (2cm and 0.5cm);
    \draw[thick] (1,1.5) ellipse (0.5cm and 2cm);
    \draw[thick] (1,-1.5) ellipse (0.5cm and 2cm);
    \draw[black] (-1.75,0) node {$C$};
    \draw[black] (1,2.75) node {$C$};
    \draw[black] (1,-2.75) node {$C$};
    \draw[black] (3.75,0) node {$C$};
    \draw (1.85,0) ellipse (1.35cm and 0.35cm);
    \draw (0.15,0) ellipse (1.35cm and 0.35cm);
    \draw (1,0.85) ellipse (0.35cm and 1.35cm);
    \draw (1,-0.85) ellipse (0.35cm and 1.35cm);
    \draw[black] (-0.75,0) node {$A$};
    \draw[black] (1,1.75) node {$A$};
    \draw[black] (1,-1.75) node {$A$};
    \draw[black] (2.75,0) node {$A$};
    
    \foreach \name/\x/\y in {v/0.9/0, u/0.9/1, w/0.1/0}
      \node[vertex] (P-\name) at (\x,\y) {~};

    \foreach \name/\x/\y in {$v$/1.1/0, $u$/1.1/1, $w$/-0.15/0}
      \draw[black] (\x,\y) node {\name};
      
    \draw [color=black] (P-u) .. controls +(-0.5,0) and +(0,0.5) .. (P-w);
    \draw [color=black, dashed] (P-w) .. controls +(0.5,0.25) and +(-0.5,-0.25) .. (P-v);
    
  \end{tikzpicture}
  \caption{An $H$-free graph $G$ with $t=s=0$ as defined in Lemma \ref{lem:noGreyTerritory}.}
  \label{figure:noGreyTerritory}
\end{figure}
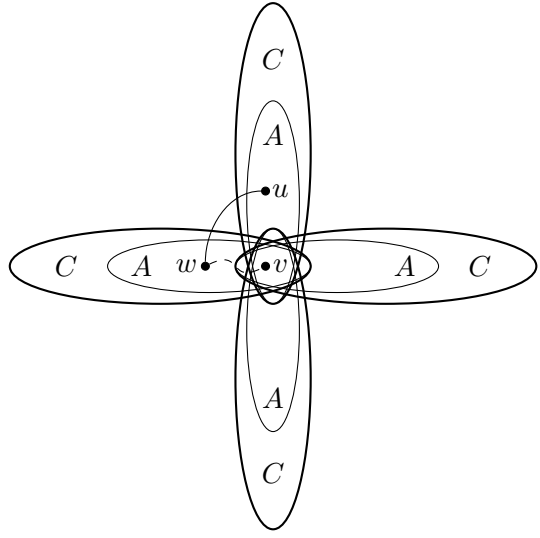

Again, to avoid creating an $H$-minor we can add at most $2 \beta_H^2 m + \beta_H s$ edges between these vertices. Since we can add at most $h-3$ edges incident to $v$, there exists an edge-maximal $H$-free graph $G'$ on $n = (c-1)m+1+t(h-1)+s$ vertices with
\[
e(G') \leq e(C)m + 2 \beta_H^2 m + \beta_H s + h-3 < m \left( \frac{h-2}{2}c - M + 2 \beta_H^2 \right ) + h(\beta_H+1).
\]
We take $G''$ to be an edge-maximal $H$-free graph on $n$ vertices consisting of $\lfloor n/(h-1) \rfloor$ disjoint copies of $K_{h-1}$ and one copy of $K_i$ for some $0 \leq i \leq h-2$. Hence we have
\[
e(G'') > \frac{n(h-2)}{2} - \frac{(h-2)^2}{2}.
\]
Therefore
\[
\begin{split}
e(G'') - e(G') & \geq ((c-1)m+1)\frac{h-2}{2} - \frac{(h-2)^2}{2} - m \left ( \frac{h-2}{2}c - M + 2 \beta_H^2 \right ) - h(\beta_H+1) \\
 & \geq  m \left ( M- \frac{h-2}{2} - 2 \beta_H^2 \right ) - \frac{(h-3)(h-2)}{2} - h(\beta_H+1) \\
 & \geq m - \frac{(h-3)(h-2)}{2} - h(\beta_H+1),
 \end{split}
\]
where the last inequality follows from \eqref{eq:boundOnM}. Consequently,
\[
\frac{e(G'') - e(G')}{n} \geq \frac{m - \frac{(h-3)(h-2)}{2} - h(\beta_H+1)}{(c-1)m+1+t(h-1)+s} \to \frac{1}{c-1}
\]
as $n \to \infty$. This completes the proof of Lemma~\ref{lem:noGreyTerritory}, and thus of Theorem \ref{thm:threeClasses}.
\end{proof}


\section{Purity with one forbidden connected minor}
\label{sec:allNonPure}

In this section we complete the proof of Theorem~\ref{thm:connpure}, showing that $K_2, K_3, K_4$ and $P_3$ are indeed the only connected graphs yielding pure $H$-free classes of graphs.

\begin{lemma}
 \label{thm:allNonPure}
 Let $H \notin \{K_2, K_3, K_4, P_3\}$ be a connected graph. Then $\Ex(H)$ is not a pure class of graphs.
\end{lemma}

By Lemma~\ref{lem.noleaf}, $\Ex(H)$ is not pure if the minimum degree $\delta(H) \geq 2$. By Lemma \ref{lem:separatingVertex} and Observation \ref{obs:claw}, $\Ex(H)$ is not pure if there is a strong separating vertex. Note that, in particular, Lemma \ref{lem:separatingVertex} and Observation \ref{obs:claw} cover all graphs $H$ such that some vertex of $H$ has at least two leaves attached to it. Hence in this section we focus on graphs $H$ with at least one leaf and with no strong separating vertex.

\begin{remark}
 In what follows, for various classes of excluded minors $H$ with $v(H) = h$ we prove that there exists some $n \in \NN$ such that $\gap_H(n) > 0$. Since we consider graphs $H$ with $\delta(H)=1$, this immediately implies that for all $k \geq 0$ we have $\gap_H(n+k(h-1)) > 0$. Indeed, a disjoint union of an edge-maximal $H$-free graph $G$ and $k$ copies of $K_{h-1}$ is again an edge-maximal $H$-free graph.
\end{remark}

\begin{lemma}
 \label{lem:twoDisjointLeaves}
 Let $H$ be a graph on $h > 5$ vertices with at least two leaves, and with no strong separating vertex. Then the class of $H$-free graphs is not pure.
\end{lemma}
\begin{proof}
 Let $G_1$ be the union of $K_{h-1}$ and an isolated vertex. Clearly $G_1$ is $H$-free and is maximal since $H$ has leaves. Also, $e(G_1) = \binom{h-1}{2}$.
 
Let $G_2$ be formed from a $K_{h-2}$ and a $K_3$ that have one vertex in common. To see that $G_2$ is $H$-free notice that the removal of the common vertex would leave no component of size at least $h-2$. Also, $G_2$ is edge-maximal $H$-free since adding an extra edge would allow us to place two leaves of $H$ in the initial $K_3$. Obviously, $e(G_2) = \binom{h-2}{2}+3$ and $e(G_1) > e(G_2)$ for all $h > 5$.
\end{proof}

\begin{observation}
\label{obs:n=4twoLeaves}
The only connected graph $H$ on $4$ vertices with at least two leaves and with no strong separating vertex is $P_4$, the path on $4$ vertices. However, let us show that $\limp(P_4) = \frac12$. Indeed, every edge-maximal $P_4$-free graph has at most one isolated vertex, thus we have $M^-_{P_4}(n) \geq \frac{n-1}{2}$. Also, a perfect matching for $n$ even, or a triangle plus a perfect matching on the remaining $n-3$ vertices for $n$ odd, is edge-maximal $P_4$-free, so $M^-_{P_4}(n) \leq \frac{n+3}{2}$.

On the other hand, any component of a $P_4$-free graph must be acyclic or unicyclic, as otherwise it would contain a $C_4$ or a \emph{bowtie graph} (two triangles with one common vertex) as a minor, thus it would not be $P_4$-free. Hence $M^+_{P_4}(n) \leq n$. Since a star on $n$ vertices guarantees $M^+_{P_4}(n) \geq n-1$, we have $\limp(P_4) = 1-\frac12 = \frac12$.
\end{observation}

\begin{observation}
\label{obs:n=5twoLeaves}
The only connected graph $H$ on $5$ vertices with at least two leaves and with no strong separating vertex consists of a triangle on $\{1,2,3\}$ and two additional edges $\{1,4\}, \{2,5\}$ (it is the so called \emph{bull graph}). Let us show that $\limp(H) = \frac12$. Since every edge-maximal $H$-free graph has at most one acyclic component, we have $M^-_{H}(n) \geq n-1$. On the other hand, for all $n \geq 5$ the cycle $C_n$ is edge-maximal $H$-free so $M^-_{H}(n) \leq n$ for all $n \geq 5$.

Let us show that $M^+_{H}(n) \leq \frac{3n}{2}$. Let $C$ be a component of size at least $5$ of an edge-maximal $H$-free graph (components of size at most $4$ trivially have the edge-to-vertex ratio at most $3/2$). Observe that for any cycle in $C$, at most one vertex of the cycle has degree higher than $2$. Otherwise we immediately can find a bull graph in $C$, or $C$ contains the \emph{diamond graph} ($K_4$ less an edge) as a subgraph (hence also the bull, since $|C| \geq 5$). Thus $C$ is obtained from a tree by adding disjoint cycles, and then identifying one vertex of the cycle with one vertex of the tree. Hence, to maximise the ratio $e(C)/v(C)$ we should take all cycles to be triangles, and the tree to be just one vertex. This gives $e(C)/v(C) \leq 3(n-1)/2n$. Therefore $M^+_{H}(n) \leq \frac{3n}{2}$.

On the other hand, a disjoint union of $\lfloor n/4 \rfloor$ copies of $K_4$ is $H$-free, so we have $M^+_{H}(n) \geq \frac{3(n-3)}{2}$. This gives $\limp(H) = \frac32-1 = \frac12$.
\end{observation}

We claim that the only graphs that remain to be checked are graphs $H$ with exactly one leaf $v$ and such that the graph $H' = H - v$ is $2$-connected. Indeed, if $\delta (H') = 1$ then either $H$ has two leaves or the unique vertex of degree $1$ in $H'$ is the neighbour $u$ of $v$ in $H$. In the latter case, let the unique neighbour of $u$ in $H'$ be $w$. Then either all the components of $H - w$ have size at most $v(H)-3$ (so $w$ is a strong separating vertex), or $H$ is a $P_4$. Since neither of these is possible, we have $\delta (H') \geq 2$. But then, if $H'$ has connectivity $1$ then clearly $H$ has a strong separating vertex. This establishes our claim.

Unfortunately, it will require several more steps to deal with the case in the claim.

\begin{lemma}
 \label{lem:cliquePlusALeaf}
 Let $H$ be a graph on $h \geq 5$ vertices consisting of a clique on $h-1$ vertices and one pendant edge. Then the class of $H$-free graphs satisfies $\limp(H) \geq \frac{h-4}{2}$.
\end{lemma}
\begin{proof}
 Let $n = m(h-1)+k$, $0 \leq k \leq h-2$. Let $G_1$ be the union of $m$ disjoint copies of $K_{h-1}$ and one copy of $K_k$. Clearly $G_1$ is edge-maximal $H$-free. Also, $v(G_1) = n = m(h-1)+k$ and
 \[
  e(G_1) = m \binom{h-1}{2}+\binom{k}{2} \leq \frac{h-2}{2} n.
 \]
 
We construct a denser $n$-vertex graph $G_2$ as follows. We start with a clique on $h-4$ vertices and a cycle $C_{n-h+4}$. We then build a complete bipartite graph between the clique and the cycle (see Figure \ref{figure:cliquePlusALeaf}). To see that $G_2$ is $H$-free note that in order to obtain a clique on $h-1$ vertices we would need to contract the cycle $C_{n-h+4}$ to a triangle, but then we would only have $h-1$ vertices left in the graph. But
\[
  e(G_2) = \binom{h-4}{2}+n-h+4 + (h-4)(n-h+4) = (h-3) n - \frac{(h-1) (h-4)}{2}.
\]
Hence
\[
 \frac{e(G_2) - e(G_1)}{n} \geq \frac{n \left (h-3 - \frac{h-2}{2} \right )-\frac{(h-1) (h-4)}{2} }{n} \to \frac{h-4}{2}
\]
 as $n \to \infty$.
\end{proof}

\begin{figure}[htb] \centering
  \begin{tikzpicture}
    \tikzstyle{vertex}=[draw,shape=circle,minimum size=5pt,inner sep=0pt]
    
    \node[vertex] (Q-1) at (0,0) {~};
    \node[vertex] (Q-2) at (0,1) {~};
    \node[vertex] (Q-3) at (0,2) {~};
    \node[vertex] (Q-4) at (0,3) {~};
    \node[vertex] (Q-5) at (0,4) {~};
    \foreach \x/\y in {1/2,2/3,3/4,4/5} {
     \draw [color=black] (Q-\x) -- (Q-\y);
     }
    \draw [color=black] (Q-5) .. controls +(-0.5,0.5) and +(-0.5,-0.5) .. (Q-1);
    \draw[black] (-1,2) node {$C_5$};
    
    \node[vertex] (P-1) at (3,3) {~};
    \node[vertex] (P-2) at (2.75,2) {~};
    \node[vertex] (P-3) at (3,1) {~};
     \draw [color=black] (P-1) -- (P-2);
     \draw [color=black] (P-1) -- (P-3);
     \draw [color=black] (P-2) -- (P-3);
    \draw[black] (3.5,2) node {$K_3$};
    
    \draw[black] (1,-0.5) node {$G_2$};
    
    \foreach \x in {1,2,3,4,5}
    \foreach \y in {1,2,3} {
     \draw [color=black] (P-\y) -- (Q-\x);
     }
     
  \end{tikzpicture}
  \caption{Graph $G_2$ as defined in Lemma \ref{lem:cliquePlusALeaf} for $h=7$ and $n=8$.}
  \label{figure:cliquePlusALeaf}
\end{figure}
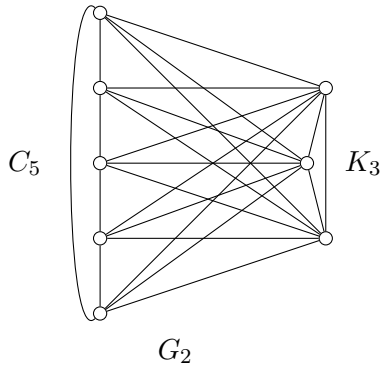

\begin{observation}
 \label{obs:pan}
 In Lemma \ref{lem:cliquePlusALeaf} we prove linear impurity of $\Ex(H)$ when the clique in $H$ contains at least $4$ vertices. Indeed, when $H$ is the \emph{pan graph} on $4$ vertices, consisting of a triangle and a pendant edge, then $\Ex(H)$ is near-pure with $\gap_H(n) = 1$ for all $n \geq 4$. To see this, observe that every connected component of an $H$-free graph is either a cycle or a tree, and an edge-maximal $H$-free graph has at most one acyclic component (in fact, this component can be any tree except a path $P_m$ on $m \geq 3$ vertices which we could close to a cycle without creating an $H$-minor).
\end{observation}

\begin{lemma}
 \label{lem:leafAdjacentToABiggie}
Let the connected graph $H$ have exactly one leaf $v$, with neighbour $u$. Let $H' = H - v$ satisfy $\delta' := \delta(H') \geq 2$, and suppose that there is a vertex $w \neq u$ in $H'$ with  $\deg_{H'}(w) = \delta'$.
Then the class $\Ex(H)$ is not pure.
\end{lemma}
\begin{proof}
By Lemma~\ref{lem:cliquePlusALeaf}, we may assume that $H'$ is not complete.  Thus $2 \leq \delta' \leq h-3$, and so $h \geq 5$.

Let $G_1$ be the graph on vertex set $[h+1]$ constructed as follows. Start with a clique on $\{1,2,\ldots,h-2\}$. Next, for $i=1,2,3$, connect the vertex $h-2+i$ to $1, 2, \ldots, \delta'-2$, as well as to $\delta'-2+i$ (see Figure \ref{figure:leafAdjacentToABiggie}). Clearly, $e(G_1) = \binom{h-2}{2} + 3(\delta'-1)$. To see that $G_1$ is $H$-free, note that it has an independent set of 3 vertices each of degree $< \delta'$, so after one edge-contraction there must still be at least two vertices of degree $< \delta'$.

Next we show that $G_1$ is edge-maximal $H$-free.  Suppose that we add an edge $e$ to $G_1$, where wlog $e$ is incident to vertex $h-1$.  There are now two cases.
(a) Suppose that $e$ is incident to $h$ or $h+1$, wlog to $h$.  Contract $e$ to form a new vertex called $w$, and place $v$ at $h+1$.  If $uw \in E(H)$ then place $u$ at vertex 1; and if not then place $u$ at vertex $\delta'+1$.
(b) Suppose that $e$ is incident to a vertex in $\{\delta',\ldots,h-2\}$.  Then $e$ is not incident to at least one of vertices $\delta', \delta'+1$, wlog the former.  Place $w$ at vertex $h-1$, place $v$ at $h$, and delete vertex $h+1$. If $uw \in E(H)$ then place $u$ at vertex 1; and if not then place $u$ at vertex $\delta'$. 

We construct the graph $G_2$ as a disjoint union of $K_{h-1}$ and the edge $\{h,h+1\}$.  Clearly $G_2$ is edge-maximal $H$-free, and $e(G_2) = \binom{h-1}{2} + 1$.
 
We have $e(G_1) \neq e(G_2)$ unless $\delta' = (h+2)/3$. Note that the smallest value of $h$ for which this could hold with both $h$ and $\delta'$ being integers is $h=7$ (which gives $\delta' = 3$).

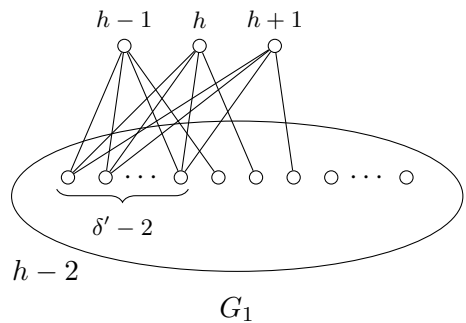
\begin{figure}[htb] \centering
  \begin{tikzpicture}
    \tikzstyle{vertex}=[draw,shape=circle,minimum size=5pt,inner sep=0pt]

    \draw (1.25,0) ellipse (3cm and 1cm);
    \node[vertex] (R-1) at (-0.25,2) {~};
    \draw[black] (-0.25,2.35) node {\footnotesize $h-1$};
    \node[vertex] (R-2) at (0.75,2) {~};
    \draw[black] (0.75,2.35) node {\footnotesize $h$};
    \node[vertex] (R-3) at (1.75,2) {~};
    \draw[black] (1.75,2.35) node {\footnotesize $h+1$};
    \node[vertex] (Q-1) at (-1,0.25) {~};
    \node[vertex] (Q-2) at (-0.5,0.25) {~};
    \draw[black] (0,0.25) node {$\cdots$};
    \node[vertex] (Q-3) at (0.5,0.25) {~};
    \node[vertex] (Q-4) at (1,0.25) {~};
    \node[vertex] (Q-5) at (1.5,0.25) {~};
    \node[vertex] (Q-6) at (2,0.25) {~};
    \node[vertex] (Q-7) at (2.5,0.25) {~};
    \draw[black] (3,0.25) node {$\cdots$};
    \node[vertex] (Q-8) at (3.5,0.25) {~};
    \draw[black] (-1.3,-1) node {$h-2$};
    \draw[black] (1.25,-1.5) node {$G_1$};
    
    \draw [decorate,decoration={brace,amplitude=5pt,mirror}, xshift=0pt,yshift=-3pt] (-1.15,0.2) -- (0.6,0.2) node [black,midway,yshift=-0.5cm] {\footnotesize $\delta'-2$};
    
    \foreach \x in {1,2,3}
     \foreach \y in {1,2,3}
      \draw (R-\x) -- (Q-\y);
    \foreach \x/\y in {1/4,2/5,3/6}
      \draw (R-\x) -- (Q-\y);
    
  \end{tikzpicture}
  \caption{Graph $G_1$ as defined in Lemma \ref{lem:leafAdjacentToABiggie}.}
  \label{figure:leafAdjacentToABiggie}
\end{figure}
 
 If we have $\delta' = (h+2)/3$ and $h > 7$, which implies that $\delta'-2 = (h-4)/3 < h-6$, then we alter our constructions of $G_1$ and $G_2$ as follows. We take the graph $G_1'$ consisting of $K_{h-2}$ together with four extra vertices $h-1, h, h+1, h+2$ such that $h-2+i$ is connected to $1, 2, \ldots, \delta'-2, \delta'-2+i$ for $1 \leq i \leq 4$ (observe that $\delta'-2+4 < h-2$). Then we have
 \[
  e(G_1') = \binom{h-2}{2} + 4(\delta'-1) = \binom{h-2}{2} + 4\frac{h-1}{3}.
 \]
 We compare $G_1'$ to $G_2'$ formed of disjoint copies of $K_{h-1}$ and $K_3$, which has $e(G_2') = \binom{h-1}{2}+3$. We then obtain
 \[
  e(G_1')-e(G_2') = \frac{h-7}{3} > 0.
 \]
 
 In the last remaining case where $h=7, \delta'=3$, we alter the construction a little bit. We build $G_1''$ on $10$ vertices, starting from a Hamiltonian cycle of edges $\{i,i+1\}$ (as usual, we identify vertex $11$ with $1$). Then we add edges to make the even vertices into a clique. Thus we have $e(G_1'') = 20$. Graph $G_1''$ is edge-maximal $H$-free for exactly the same reasons as our previous constructions: the odd vertices all have degree $2 < \delta'$ and form an independent set, while the union of the neighbourhoods of any two of them has size either $3$ or $4$. We take $G_2''$ to be a disjoint union of $K_6$ and $K_4$, so clearly $e(G_2'') = 15+6 = 21$. This completes the proof of the lemma.
\end{proof}

\begin{lemma}
 \label{lem:leafAdjacentToATiny}
 Let $H$ be a graph on $h \geq 6$ vertices with exactly one leaf $v$ and such that the graph $H' = H - v$ has connectivity $k$ for some $2 \leq k \leq h-4$. Also, let the unique neighbour of $v$ in $H$ be $u$. If $\deg_H(u) = \delta(H')+1$ then the class $\Ex(H)$ is not pure.
\end{lemma}
\begin{proof}
 We use a similar construction as in Lemma \ref{lem:noLeaves}. Let $A$ and $B$ be $(h-2)$-sets with $|A \cap B|=k-1$.  Let $G_1$ be the union of a clique on $A$ and a clique on $B$. This graph clearly has $|A \cup B| =2h-k-3 \geq h+1$ vertices and $2\binom{h-2}{2}-\binom{k-1}{2}$ edges. To see why $G_1$ is $H$-free we note that it is $H'$-free, since we cannot have a model of $H'$ within $A$ or within $B$.
 
 On the other hand, adding a single edge to $G_1$ and contracting it gives us a graph on at least $h$ vertices consisting of a union of two cliques on $h-2$ vertices each that overlap in $k$ vertices. Let us show that this graph is not $H$-free. We can obviously find $H'$ in this graph as a subgraph because $H'$ has $h-1$ vertices and connectivity (exactly) $k$. The only time we need to worry about being able to add the leaf $v$ to our minor is when all but one of the vertices of $H'$ are located in $A$ and only one in $B \setminus A$ (or vice-versa). But then that one vertex (say vertex $x$) would have degree at most $k$ in $H'$, so $\deg_{H'}(x)= \delta(H') = k$, and now we can place vertex $u$ at $x$.
 
 We take $G_2$ to be a disjoint union of $K_{h-1}$ and $K_{h-k-2}$, which is clearly seen to be edge-maximal $H$-free. We have $e(G_2) = \binom{h-1}{2}+\binom{h-k-2}{2}$. The only integer solutions to $e(G_1) = e(G_2)$ are $h=1, k=0$ and $h=5, k=2$. This completes the proof.
\end{proof}
The next lemma fills one of the gaps left by Lemma \ref{lem:leafAdjacentToATiny}.
\begin{lemma}
 \label{lem:leafPlusAlmostClique}
 Let $H$ be a graph on $h \geq 6$ vertices with exactly one leaf $v$ and such that the graph $H' = H - v$ has connectivity $h-3$. Then the class $\Ex(H)$ satisfies $\limp(H) \geq \frac{h-5}{2}>0$.
\end{lemma}
\begin{proof}
For each $m \geq 2$, let $n = h-4+2m$ and let the $n$-vertex graph $G_1$ be the union of $m$ cliques, each on $h-2$ vertices, that overlap in a common set of $h-4$ vertices (see Figure \ref{figure:leafPlusAlmostClique}). As in Lemma \ref{lem:leafAdjacentToATiny}, $G_1$ is $H$-free and has size $e(G_1) = \binom{h-4}{2} + m(2(h-4)+1)$. Let $G_2$ be a disjoint union of (as many a possible) cliques on $h-1$ vertices and possibly one smaller clique containing the remaining $k$ vertices, where $0 \leq k \leq h-2$. Then $G_2$ is edge-maximal $H$-free.
 
It is easy to see that, as $n \to \infty$, we have $e(G_1) = (1+o(1))(2(h-4)+1)n/2$, while $e(G_2) = (1+o(1))(h-2)n/2$. Thus
 \[
  \frac{e(G_1)-e(G_2)}{n} \to \frac{h-5}{2},
 \]
 as desired. Note that we do not need $G_1$ to be edge-maximal here, and so for $n = h-4+2m+1$ we can just take $G_1$ plus an isolated vertex.
\end{proof}

\begin{figure}[htb] \centering
  \begin{tikzpicture}
    \tikzstyle{vertex}=[draw,shape=circle,minimum size=5pt,inner sep=0pt]

    \draw (0,0) ellipse (1cm and 1cm);
    \draw[black] (0,0) node {$h-4$};
    
    \foreach \name/\x/\y in {11/-2/0.5, 12/-2/-0.5, 21/2/0.5, 22/2/-0.5, 31/0.5/2, 32/-0.5/2, 41/0.5/-2, 42/-0.5/-2}
      \node[vertex] (P-\name) at (\x,\y) {~};
    \foreach \x/\y in {11/12, 21/22, 31/32, 41/42}
      \draw (P-\x) -- (P-\y);
    \foreach \y in {0.4, 0.15, -0.1}
      \draw (P-11) -- (-0.6,\y);
    \foreach \y in {-0.4, -0.15, 0.1}
      \draw (P-12) -- (-0.6,\y);
    \foreach \y in {0.4, 0.15, -0.1}
      \draw (P-21) -- (0.6,\y);
    \foreach \y in {-0.4, -0.15, 0.1}
      \draw (P-22) -- (0.6,\y);
    \foreach \x in {0.4, 0.15, -0.1}
      \draw (P-31) -- (\x,0.6);
    \foreach \x in {-0.4, -0.15, 0.1}
      \draw (P-32) -- (\x,0.6);
    \foreach \x in {0.4, 0.15, -0.1}
      \draw (P-41) -- (\x,-0.6);
    \foreach \x in {-0.4, -0.15, 0.1}
      \draw (P-42) -- (\x,-0.6);
    
  \end{tikzpicture}
  \caption{Graph $G_1$ with $m=4$ as defined in Lemma \ref{lem:leafPlusAlmostClique}.}
  \label{figure:leafPlusAlmostClique}
\end{figure}
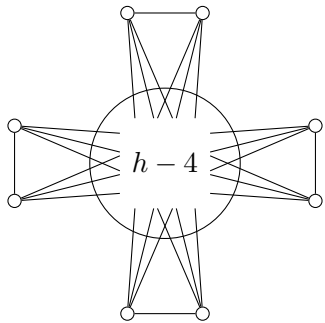

\begin{observation}
 \label{obs:h=5oneLeaf}
 The last remaining graphs that we need to consider are the connected graphs $H$ on $5$ vertices that have exactly one leaf $v$ and are such that $H-v$ is 2-connected but is not a complete graph. Up to isomorphism, there are exactly three such graphs $H$, and they each give rise to classes satisfying $\limp(H)\geq \frac12$. Consider $n>4$. In each case, our `dense' example is the disjoint union of $\lfloor n/4 \rfloor$ copies of $K_4$, together with a copy of $K_t$ where $t = n-4\lfloor n/4 \rfloor$ if $4 \nmid n$, which is an edge-maximal $H$-free graph with $3n/2+O(1)$ edges.

 \begin{enumerate}
  \item Let $H_1$ be $C_4$ with an added leaf. Then $C_n$ is an edge-maximal $H_1$-free graph with $n$ edges. Hence $\limp(H_1)\geq \frac12$.
  \item Let $H_2$ be a diamond ($K_4$ minus an edge), with an added leaf adjacent to a vertex of degree $2$ of the diamond. Then the graph obtained from $C_{n-1}$ by adding one vertex and joining it to two adjacent vertices on the cycle is an edge-maximal $H_2$-free graph with $n+1$ edges; and it follows that $\limp(H_2)\geq \frac12$.  
  \item Let $H_3$ be a diamond ($K_4$ minus an edge), with an added leaf adjacent to a vertex of degree $3$ of the diamond. Then the graph obtained from $K_4$ by subdividing one edge $n-4$ times (or equivalently, from $C_{n-1}$ by adding a vertex and joining it to three consecutive vertices on the cycle) is an edge-maximal $H_3$-free graph with $n+2$ edges; and it follows $\limp(H_3) \geq \frac12$.
 \end{enumerate}
\end{observation}

\begin{remark}
 In fact, it can be shown that $\limp(H_1) = 1/2$, $\limp(H_2) = 1$ and $\limp(H_3) = 2/3$ (see Appendix \ref{app:v5delta1}).
\end{remark}

This completes the proof of Lemma~\ref{thm:allNonPure}, and thus of Theorem~\ref{thm:connpure}.


\section{Forbidding several minors or disconnected minors}
\label{sec:generalizations}

We start this section by generalising Lemma \ref{lem:noLeaves} to a case where we may have more than one excluded minor, and the excluded minors need not be connected. For our proof to work, the forbidden set $\cH$ needs to satisfy specific and rather strict conditions.   Roughly, we require that one component of one excluded minor is `smallest' in several senses.  However, cases like $\cH = \{m C_h\}$ (that is, $m$ disjoint copies of the cycle $C_h$) for $h \geq 4$, or $\cH = \{ K_{2,3}, C_5 \}$, can be dealt with using the following result, which shows that in these cases the classes $\Ex(\cH)$ are linearly impure. 

\begin{lemma}
\label{lem:multipleMinors}
Let $\cH = \{H_1, H_2, \ldots, H_m\}$ be a set of $m \geq 1$ excluded minors.  Let $t_1,\ldots,t_m$ be positive integers.  For each $1 \leq i \leq m$, let $H_i = \bigcup_{j=1}^{t_i} H_i^j$; that is, let each graph $H_i$ be a disjoint union of connected graphs $H_i^j$ for $1 \leq j \leq t_i$. Assume that the following conditions hold:
\begin{enumerate}
 \item $v(H_1) = \min_{1 \leq i \leq m} v(H_i)$ and $v(H_1^1) = \min \left \{ v \left  (H_i^j \right ): 1 \leq i \leq m, 1 \leq j \leq t_i \right \} :=h$.
 \item $\delta(H_1^1) = \min_{1 \leq i \leq m} \delta(H_i) := \delta$ and $\delta$ satisfies $2 \leq \delta \leq v \left (H_1^1 \right )-2$.
 \item Taking $k_i^j$ to be the connectivity of $H_i^j$ we have $k_1^1 = \min \left \{ k_i^j: 1 \leq i \leq m, 1 \leq j \leq t_i \right \} : = k$.
\end{enumerate}
Then we have $\limp(\cH) \geq \frac{1}{2h}$.
\end{lemma}
\begin{proof}
The proof of this lemma is nearly identical to the proof of Lemma \ref{lem:noLeaves} when we take $H = H_1^1$. We amend the constructions of graphs $G_1$ and $G_2$ by adding to both of them a clique of size $v(H_1)-1$ and identifying $k-1$ vertices of that new clique with the ``small cut'' $I$ consisting of the central $k-1$ vertices in the previously built graphs. By our assumptions on $\cH$, these graphs are $\cH$-free, and adding an arbitrary edge to the graph creates an $H_1$-minor: we trivially find the graphs $H_1^2, \ldots, H_1^{t_1}$ in the ``large'' clique on $v(H_1)-1$ vertices (if $t_1 \geq 2$), and $H_1^1$ is created like $H$ was in Lemma \ref{lem:noLeaves}.
\end{proof}

\begin{remark}
\label{rem:disjointTriangles}
One interesting case that is not covered by Lemma \ref{lem:multipleMinors} is $\Ex(mK_3)$, i.e., the class of graphs with $m$ disjoint triangles excluded for some $m \geq 2$. However, building on the work of Corradi and Hajnal \cite{independentCircuits} on the number of disjoint cycles in graphs of given density, it was observed in \cite{sparseMinors} that every graph $G$ with $e(G) \geq (2m-1)v(G)$ contains $mK_3$ as a minor. Moreover, this bound is asymptotically sharp as demonstrated by the complete bipartite graph $G = K_{2m-1,n-2m+1}$.

On the other hand, any maximal $mK_3$-free graph can have at most one acyclic component, so $M^-_{mK_3}(n) \geq n-1$, and by analysing $G$ constructed from $K_{3m-1}$ by adding $n-3m+1$ pendant edges we see that $M^-_{mK_3}(n) = n+O(1)$. Hence we can conclude that $\limp(mK_3)= 2m-2$ for all $m \geq 2$.
\end{remark}

So far we have seen only two graphs $H$ such that the class $\Ex(H)$ is near pure. Namely, this happens when $H$ is the claw or the pan graph. However, in both cases $\gap_H(n) \leq 1$ and it is unclear whether there are more connected graphs $H$ such that $\Ex(H)$ is near pure, and if the answer to that question is positive, whether $\gap_H(n)$ can take arbitrarily large values (or in fact, any value larger than $1$). This is the case when we forbid more complex sets of graphs. In the following proposition we only take $t \geq 16$ to avoid complications in the statement that would make the conclusions more difficult to observe.
\begin{prop}
\label{prop:excludingStars}
Let $t \geq 16$ be an integer, and let $\cH = \{K_{1,t}, 2K_{1,3} \}$. Then the class $\Ex(\cH)$ is near-pure with $t-10 \leq \gap_{\cH}(n) \leq t-1$.
\end{prop}
\begin{proof}
We first claim that for all $n \geq 1$, every edge-maximal $\cH$-free graph $G$ satisfies $e(G) \geq n-1$. Indeed, every $\cH$-free graph must have at most one component that is not a cycle nor a path to avoid creating a $2K_{1,3}$-minor. Consequently, a maximal $\cH$-free graph has at most one acyclic component because we could connect one of the endpoints of any path to a leaf of any other tree without creating any of the forbidden minors.

Now let $G$ be an edge-maximal $\cH$-free graph on $n \geq 4$ vertices. Let $\Delta = \Delta(G)$ denote the maximum degree of a vertex in $G$. Clearly $\Delta \geq 3$.  We consider three cases -- when $3 \leq \Delta \leq 5$, $\Delta=6$ and $\Delta \geq 7$.  In each case, let $v$ be a vertex of degree $\Delta$, let $V_i$ denote the set of vertices at distance  $i$ from $v$ in $G$, and let $W_2$ denote $\bigcup_{i \geq 2} V_i$, the set of vertices at distance at least $2$ from $v$ (recall that $G$ might not be connected).

Suppose first that $3 \leq \Delta \leq 5$.  Each vertex in $V_1$ can have at most 2 edges to $V_2$ (this is immediate if $\Delta = 3$, while for $\Delta \geq 4$ follows from the fact that otherwise we have a $2K_{1,3}$ minor), so there are at most $2\Delta$ edges between $V_1$ and $V_2$.  Similarly, each vertex in $W_2$ can have at most 2 edges to vertices in $W_2$.  Hence the degree sum is at most 
\[
 (\Delta+1) \cdot \Delta +2\Delta +(n-\Delta-1) \cdot 2=2n+(\Delta+1)\Delta-2 \leq 2n+28.
\]
If $\Delta = 6$ then each vertex in $W_2$ has degree at most 2, so the degree sum is at most $7 \cdot 6 + (n-7)\cdot 2=2n+28$. (Observe that the disjoint union of $K_7$ and $C_{n-7}$ achieves this bound.)  If $\Delta \geq 7$ then also each vertex in $V_1$ has degree at most 3, so the degree sum is at most 
\[
 \Delta+ \Delta \cdot 3 + (n-\Delta-1) \cdot 2 = 2n + 2 \Delta -2;
\]
and since also $\Delta \leq t-1$ this is at most $2n+2t-4$. Thus for $\Delta \leq 6$ we have $e(G) \leq n+14$, and for $\Delta \geq 7$ we have $e(G) \leq n+t-2$.  Hence, since $t \geq 16$, we always have $e(G) \leq n+t-2$.  It follows that $\gap_{\cH}(n) \leq n+t-2 -(n-1) \leq t-1$.

The upper bound $n+t-2$ is achieved.  Let $G'$ be the disjoint union of the $t$-vertex wheel (a~$C_{t-1}$ plus a central vertex) and $C_{n-t}$.   Then $G'$ is (edge-maximal) $\cH$-free, and $e(G')=n+t-2$.

We now find a much smaller edge-maximal $\cH$-free graph. Start with $K_5$, choose two vertices $u$ and $v$ in the $K_5$ and add two new vertices $x$ and $y$, both adjacent to both of $u$ and $v$ (this gives the total of $10+4$ edges so far). Finally we add $n-7$ vertices which form a path of $n-6$ edges between $x$ and $y$. The resulting graph is edge-maximal $\cH$-free with $n+8$ edges.  Hence $\gap_{\cH}(n) \geq n+t-2 -(n+8) =t-10$.
\end{proof}


\section{Concluding remarks and open problems}

When the connected graph $H$ satisfies $\delta(H)=1$ then a natural example of a leaf-and-edge-maximal $H$-free graph is a union of disjoint copies of $K_{h-1}$, where $h=v(H)$.  It often turns out to be a ``dense'' example of such a graph, though in some cases we can find denser $H$-free graphs (see, e.g., Lemmas \ref{lem:cliquePlusALeaf} and \ref{lem:leafPlusAlmostClique}). In general, it appears that the graphs with minimum degree $1$ can cause us most trouble analysing their purity, as illustrated in the following example.

\begin{example}
Let the graph $H$ on $h \geq 6$ vertices consist of a clique on $h-2$ vertices and two pendant, non-incident edges. Two obvious examples of edge-maximal $H$-free graphs are a union of disjoint cliques each on $h-1$ vertices, and a union of cliques each on $h-2$ vertices that share one common vertex. It is easy to check that in both cases the density of these graphs tends to $(h-2)/2$ as the number of cliques constituting them tends to infinity. As finding other edge-maximal $H$-free graphs appears non-trivial, this might suggest that $\Ex(H)$ is near-pure.

This is, however, not true and the following sparse construction, after comparing with the disjoint union of copies of $K_{h-1}$ (together with a smaller clique if necessary), will show that we have
\begin{equation} \label{eqn.twoleaves}
 \limp(H) \geq \frac{h-4}{2}
\end{equation}
for all $h \geq 6$. Let $G'$ be a subdivision of $K_{h-2}$, obtained from $K_{h-2}$ by subdividing every edge at least once. Let $H^-$ be $H$ less a leaf (that is, $K_{h-2}$ plus one pendant edge). Clearly, adding an edge joining two vertices created through subdivisions of the same edge of $K_{h-2}$ creates an $H^-$-minor in $G'$. In fact, by case analysis, it is easy to check that $G'$ is leaf-and-edge-maximal $H^-$-free (it is enough to check it for $h=6$ because the edge we add to $G'$ can be ``wrapped'' in a $K_4$ containing it). Importantly, when we add an edge to $G'$ then we always have at least two choices of an original vertex of $K_{h-2}$ to which we can attach a leaf of the $H^-$-minor (see Figure \ref{figure:cliquePlusTwoLeafs}). For $n$ large enough, it is then enough to take a union of two such (not necessarily identical) subdivisions of $K_{h-2}$ of sizes that sum up to $n+1$, and connect them by  picking an original vertex of $K_{h-2}$ from each subdivided graph and identifying them. The resulting graph is edge-maximal $H$-free with density tending to $1$ as $n$ tends to infinity. This establishes~\eqref{eqn.twoleaves}, and completes the example.

\begin{figure}[htb] \centering
  \begin{tikzpicture}[scale = \subdivisionscale]
    \tikzstyle{vertex}=[draw,shape=circle,minimum size=5pt,inner sep=0pt]
    
    \node[vertex] (P-1) at (0,0) {~};
    \node[vertex] (P-2) at (0,2) {~};
    \node[above = 0.6em, left = 0.1em] () at (0,2) {$1$};
    \node[vertex] (P-3) at (2,2) {~};
    \node[vertex] (P-4) at (2,0) {~};
    \node[below = 0.6em, right = 0.1em] () at (2,0) {$2$};
    
    \node[vertex] (Q-11) at (0,2/3) {~};
    \node[vertex] (Q-12) at (0,4/3) {~};
    \node[vertex] (Q-21) at (1/2,0) {~};
    \node[vertex] (Q-22) at (1,0) {~};
    \node[vertex] (Q-23) at (3/2,0) {~};
    \node[vertex] (Q-3) at (2,1) {~};
    \node[right = 0.2em] () at (2,1) {$b$};
    \node[vertex] (Q-4) at (1,2) {~};
    \node[above = 0.2em] () at (1,2) {$a$};
    \node[vertex] (Q-5) at (1,1) {~};
    \node[left = 0.6em, below = 0.2em] () at (1,1) {$c$};
    \node[vertex] (Q-6) at (2.75,-0.75) {~};
    
    \foreach \x/\y in {1/11, 1/21, 2/12, 2/4, 2/5, 3/4, 3/3, 4/23, 4/3, 4/5} {
     \draw [color=black] (P-\x) -- (Q-\y);
     }
    \draw [color=black] (Q-11) -- (Q-12); 
    \draw [color=black] (Q-21) -- (Q-22); 
    \draw [color=black] (Q-22) -- (Q-23); 
    \draw [color=black] (P-1) .. controls +(1,-1) and +(-0.5,-0.5) .. (Q-6);
    \draw [color=black] (P-3) .. controls +(1,-1) and +(0.5,0.5) .. (Q-6);
    
    \draw [color=black, dashed] (Q-3) -- (Q-4);
     
  \end{tikzpicture}
  \caption{Graph $G'$ for $h=6$ which is a subdivision of $K_4$. When, for example, we add the edge $\{a,b\}$ to $G'$, we can contract $\{1,a\}$ and $\{2,b\}$ and delete either $\{1,c\}$ or $\{2,c\}$, hence finding a minor consisting of a $K_4$ and a pendant edge attached to either $1$ or $2$.}
  \label{figure:cliquePlusTwoLeafs}
\end{figure}
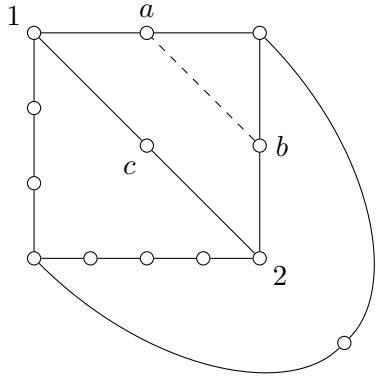
\end{example}

Let us recall the definition of the set
\[
 E_\cH(n) = \{ e(G) : v(G) = n \text{ and } G \text{ is an edge-maximal $\cH$-free graph} \}.
\]
The main objects of study of this paper were the extreme values of the set $E_\cH(n)$, i.e., $M^-_\cH(n)$ and $M^+_\cH(n)$. However, once we know that $\Ex(\cH)$ is not pure (i.e., that $M^-_\cH(n) \neq M^+_\cH(n)$ for some $n$), we can ask additional questions about the structure of $E_\cH(n)$. As a test case, let us consider $\cH = \{K_5\}$.

Recall again the result of Wagner, who proved that edge-maximal $K_5$-free graphs are obtained as $2$- or $3$-clique-sums of planar graphs and of the Wagner graph $W_8$ (the sums must be ``maximal'', in particular, we only take a $2$-clique-sum of two graphs along an edge if that edge in not in any triangle in at least one of those graphs).

Consequently, taking clique-sums of only planar graphs, always leads to building edge-maximal $K_5$-free graphs on $n$ vertices and $3n-6$ edges. Therefore, the first interesting case is $n=8$. The only possible edge-numbers of edge-maximal $K_5$-free graphs on $8$ vertices are $12$ (the Wagner graph $W_8$) and $18$ ($3$-clique-sums of planar graphs). For $n=9$ these edge-numbers are $14$ ($W_8$ glued to a triangle) and $21$, while for $n=10$ it can be $16$ ($W_8$ plus two triangles glued to different edges of $W_8$), $17$ ($W_8$ and $K_4$ glued along an edge) or $24$. Continuing this way, for $n=14$, we can build edge-maximal $K_5$-free graphs with any number of edges between $23$ and $29$, as well as $36$.

More generally, taking $0 \leq i \leq 5$ and $n = 6k+2+i$ large, we have $M^-_{K_5}(n) = \frac{11(n-2-i)}{6}+1+2i$, and $E_{K_5}(n)$ contains all values between $M^-_{K_5}(n)$ and $3n-13$ (obtained, e.g., using one copy of $W_8$ glued along an edge with a triangulation on $n-6$ vertices), as well as $3n-6$. Hence in general we don't have $E_\cH(n)$ forming an interval, but do we always have $\gap_\cH(n) - |E_\cH(n)| = O(1)$, or at least is it always the case that if $\Ex(\cH)$ is linearly impure then $|E_H(n)| / \gap_H(n) \to 1$ as $n$ tends to infinity?

We have determined the complete list of four connected graphs $H$ leading to pure minor-closed classes $\Ex(H)$. For connected $H$ we also know that $\Ex(H)$ is linearly impure if
\begin{itemize}
 \item $\delta(H) \geq 2$, see Lemma \ref{lem:noLeaves}, or
 \item $H$ has a strongly separating vertex (except for the claw $K_{1,3}$), see Lemma \ref{lem:separatingVertex}, or
 \item $H$ is the path $P_4$ (Observation \ref{obs:n=4twoLeaves}), the bull graph (Observation \ref{obs:n=5twoLeaves}), a clique on at least four vertices with an additional one  (Lemma \ref{lem:cliquePlusALeaf}) or two leaves (see the discussion at the beginning of this section), or
 \item $H$ consists of a clique on at least five vertices minus a matching, plus a pendant edge, see Lemma \ref{lem:leafPlusAlmostClique}, or
 \item $H$ is one of the three graphs discussed in Observation \ref{obs:h=5oneLeaf}.
\end{itemize}

Additionally, we know that $\Ex(H)$ is near-pure with $\gap_H(n) = 1$ if $H$ is the claw (Observation \ref{obs:claw}) or the pan graph (Observation \ref{obs:pan}). What about the remaining connected graphs $H$ which are not pure? Are there any more near-pure minor-closed classes $\Ex(H)$ for some connected graph $H$? Can we find an example such that $\gap_H(n) \geq 2$ for some $n$?

We defined $\limp(\cH) = \liminf_{n \to \infty} \gap_{\cH}(n)/n$. Theorem~\ref{thm:AddableLimit} says that $\gap_{\cH}(n)/n$ tends to a limit if all graphs in $\cH$ are $2$-connected, so that in this case we could define $\limp(\cH)$ as the limit of $\gap_{\cH}(n)/n$. Do we always have $\gap_{\cH}(n)/n \to \limp(\cH)$?

Finally, what about minor-closed classes with two or more connected excluded minors, whose analysis we started in Section \ref{sec:generalizations}: which are the pure classes, and are all such classes pure or near-pure or linearly impure? For example, the classes $\Ex(K_5, K_{3,3})$ of planar graphs, $\Ex(K_3, K_{1,3})$ of `forests of paths', $\Ex(2K_2, K_{3})$ of a star and isolated vertices, $\Ex(\text{Diamond}, \text{Bowtie})$ of graphs consisting of unicyclic and acyclic components, and $\Ex(K_4, K_{2,3})$ of outerplanar graphs are all pure; while for all $t \geq 5$, the class $\Ex(C_t, K_{1,3})$ where each component is a path or a short cycle, is near-pure with $\gap(n) = 1$ for all $n \geq \max\{t,6\}$ (two examples of $\{C_t, K_{1,3}\}$-free edge-maximal graphs are a path on $n$ vertices and $n-1$ edges, and a union of disjoint copies of $C_3$ and $C_4$ with total of $n$ vertices and $n$ edges, which exists for all $n \geq 6$ by Fact \ref{fact:frobeniusNumber}). Note that $\Ex(C_4, K_{1,3})$ is an interesting case with $\gap(3k) = 1$ for all $k \geq 2$, and $\gap(n) = 0$ otherwise. Obviously, similar questions could be asked about excluding disconnected minors.

\paragraph{Acknowledgements} We would like to thank Andrius Vaicenavicius for stimulating discussions during the course of this work. We would also like to thank the referee for careful comments.

\begin{appendices}
 \section{Connected graphs $H$ on $5$ vertices with $\delta(H)=1$}
 \label{app:v5delta1}

In this appendix we refine the analysis in Observation \ref{obs:h=5oneLeaf} and, in the following three propositions, we study the purity of the connected graphs $H$ on $5$ vertices that have exactly one leaf $v$ and are such that $H-v$ is 2-connected but is not a complete graph.

\begin{prop}
 \label{prop:C4plusLeaf}
 Let $H$ be $C_4$ with an added leaf. Then $\gap_{H}(n)= \tfrac12 n +O(1)$ and so $\limp(H) = \frac12$.
\end{prop}
\begin{proof}
 Any edge-maximal $H$-free graph has at most one acyclic component, thus $M^-_H(n) \geq n-1$.  Also, for $n \geq 6$, the cycle $C_{n-1}$ together with an isolated vertex is edge-maximal $H$-free graph with $n-1$ edges. Hence $M^-_H(n) = n-1$ for $n \geq 6$.
 
 On the other hand, a disjoint union of $\lfloor n/4 \rfloor$ copies of $K_4$, together with a copy of $K_t$ where $t = n-4\lfloor n/4 \rfloor$ if $4 \nmid n$, is an edge-maximal $H$-free graph with $3n/2+O(1)$ edges. Hence we have $M^+_H(n) \geq 3n/2+O(1)$. It remains to show that every $H$-free graph $G$ has $e(G) \leq 3n/2+O(1)$.
 
 Let $C$ be a component of $G$. If $v(C) \leq 4$ then clearly $e(C) \leq 3v(C)/2$. If $v(C) > 4$ and $C$ is not $C_4$-free then $C$ must be a cycle, hence $e(C) = v(C)$. Finally, if $C$ is $C_4$-free then each block of $C$ is an edge or a triangle, so $e(C) \leq (3v(C)-3)/2$. (Recall that a \emph{block} of a graph is a maximal connected subgraph that has no cut-vertex.) This implies that $M^+_H(n) \leq 3n/2$, and so $M^+_H(n) = 3n/2 +O(1)$.  Thus we have $\gap_{H}(n)= n/2 +O(1)$.
\end{proof}

\begin{prop}
 \label{prop:diamondPlusLeaf1}
 Let $H$ be a diamond ($K_4$ minus an edge), with an added leaf adjacent to a vertex of degree $2$ of the diamond. Then $\gap_H(n) = n-3$ for each $n \geq 6$ and so $\limp(H) = 1$.
\end{prop}
\begin{proof}
Let $G$ be an edge-maximal $H$-free graph.  Then at most one component has at most one cycle.  Further any acyclic component has at most two vertices.  So $M^-_H(n) \geq n$ for $n \geq 3$. Now let $n \geq 6$ and $G$ be formed from an $(n-3)$-cycle and three non-incident pendant edges.  It is easy to check that $G$ is an edge-maximal $H$-free graph.  Thus $M^-_H(n) = n$ for $n \geq 6$.
 
 If $G$ is a $K_{2,n-2}$ with an extra edge joining the two vertices in the small class then $G$ is edge-maximal $H$-free, and so $M^+_H(n) \geq 2n-3$. Let us show that any $H$-free graph $G$ with $v(G) \geq 5$ satisfies $e(G) \leq 2v(G) -3$.
 
 Let $C$ be a component of $G$. If $C$ is $K_4$-free then we are done. Hence assume that $C$ is not $K_4$-free. Hence it has a subgraph homeomorphic to $K_4$. We observe that this subgraph must be spanning $C$ (as otherwise we would have an $H$-minor). Also, if any of the edges of $K_4$ were subdivided, it would also create an $H$-minor. Therefore $C$ must be a $K_4$. Hence $M^+_H(n)=2n-3$ for each $n \geq 2$, except $M^+_H(4)=6$. Thus $\gap_H(n) = n-3$ for each $n \geq 6$.
\end{proof}

\begin{prop}
 \label{prop:diamondPlusLeaf2}
 Let $H$ be a diamond, with an added leaf adjacent to a vertex of degree $3$ of the diamond. Then $\gap_{H}(n)= \tfrac23 n+O(1)$ and so $\limp(H) = \tfrac23$.
\end{prop}
\begin{proof}
We argue as before for $M^-_H(n)$.
Let $G$ be an edge-maximal $H$-free graph.  Then at most one component has at most one cycle.  Further any acyclic component has at most two vertices.  So $M^-_H(n) \geq n$ for $n \geq 3$. Now let $n \geq 7$ and $G$ be formed from an $(n-3)$-cycle and three non-incident pendant edges.  It is easy to check that $G$ is an edge-maximal $H$-free graph.  Thus $M^-_H(n) = n$ for $n \geq 7$.

 Now, let $n = 3k + 1 + i \geq 6$ for some $i \in \{0,1,2\}$. Then a graph obtained as follows: take $k$ copies of a diamond graph ($K_4$ minus an edge) and a $K_{i+1}$ (if $i > 0$) and connect them into one graph by identifying one of the vertices of degree $2$ in every diamond, and and arbitrary vertex of the $K_{i+1}$. The resulting graph is edge-maximal $H$-free with $5(n-1-i)/3+\binom{i+1}{2}$ edges, hence $M^+_H(n) \geq 5n/3+O(1)$. Proving the inequality in the other direction shall require a bit more work.
 
 Let $G$ be an edge-maximal $H$-free graph. Suppose $G$ has a component $C$ with a $K_4$-minor. Then $C$ must be a subdivision of $K_4$. Indeed, $C$ has a subgraph which is a subdivision of $K_4$ and it is easy to see that this subgraph must be spanning. By case analysis we may check that no edge can be added, or we would obtain $H$ as a minor. So we have $e(C) \leq 3 v(C)/2$. 

 Now, let $C'$ be a block of $G$ with no $K_4$-minor, and suppose that $C'$ is not just an edge or cycle. Then $C'$ contains a subgraph $D$ which is a subdivision of the diamond graph (since $C'$ is $2$-connected). Then $D$ consists of two vertices, $a$ and $b$, and three internally vertex disjoint $ab$ paths $P_1, P_2, P_3$ (one of which may be just an edge).

 We claim that $D$ is spanning in $C'$. For suppose there is a vertex of $C'$ not in $D$. Then, since $C'$ is $2$-connected, there are distinct vertices $u$ and $v$ in $D$ and a path $P_{out}$ of length at least $2$ between them outside $D$.  Let $u_{out}$ be the neighbour of $u$ on $P_{out}$.  Clearly $\{u,v\} \cap \{a,b\} = \emptyset$ (or we would obtain $H$ as a minor).  Also, $u$ and $v$ must be on the same path $P_i$ (or we would obtain a $K_4$-minor). Now $u$ is not adjacent in $P_i$ to at least one of $a, b$, say not adjacent to $b$. If we contract the segment of $P_i$ between $u$ and $a$ to a new vertex $a'$, we obtain a copy of $D$ with three paths between $a'$ and $b$, plus the edge $a'u_{out}$; and so we have a minor $H$, a contradiction.

 We now know that $D$ is spanning $C'$. We claim that $C'$ is in fact equal to $D$. Indeed, suppose that there is an extra edge $xy$ in $C'$. This edge cannot be between internal vertices on distinct paths $P_i$ (or we get $K_4$ as a minor).  If $\{x,y\}=\{a,b\}$ then each path $P_i$ has length at least 2 and we find an $H$-minor. If $\{x,y\} \cap \{a,b\}$ is a singleton, wlog $x=a$, then we find $H$ (with the 'extra' vertex being the neighbour of $a$ on the path $P_i$ from $a$ to $y$). The last case is when $x$ and $y$ are internal vertices of the same path $P_i$, wlog appearing along $P_i$ in the order $a,x,y,b$ (with some vertices in between, in particular between $x$ and $y$). Now we can contract the segment of $P_i$ between $a$ and $x$, and we are back in the case when $\{x,y\} \cap \{a,b\}$ is a singleton.

 We have now seen that $C' = D$.  Hence $e(C') \leq (5/3)(v(C')-1)$. Thus each block $B$ of $G$ with no $K_4$-minor has $e(B) \leq (5/3)(v(B)-1)$. Hence each component $C$ of $G$ with no $K_4$-minor has $e(C) \leq 5(v(C)-1)/3$. We may also have components $\tilde{C}$ which are subdivisions of $K_4$, and then $e(\tilde{C})< 5v(\tilde{C})/3$. Hence $e(G)< 5v(G)/3$, so $M^+_H(n)< 5n/3$. We have now seen that  $M^+_H(n)= 5n/3+O(1)$, so $\gap_{H}(n)= 2n/3+O(1)$. This completes the proof of the Proposition.
\end{proof}

\end{appendices}

\bibliographystyle{elsarticle-num}

\end{document}